\documentclass[12pt,leqno,a4paper]{article}
\usepackage{amsthm}
\usepackage{indentfirst}
\usepackage{amsmath}
\usepackage{mathrsfs}
\usepackage{paralist}
\usepackage{amssymb,enumerate}
\usepackage{hyperref}
\usepackage{framed}
\usepackage{extarrows}
\usepackage[hmargin=3cm,vmargin=3cm]{geometry}
\usepackage{url}
\usepackage{tikz-cd}
\usetikzlibrary{cd}
\usetikzlibrary{arrows.meta}
\usepackage{commutative-diagrams}

\theoremstyle{theorem}

\swapnumbers
\theoremstyle{theorem}
\newtheorem{thm}{Theorem}[section]

\newtheorem{lem}[thm]{Lemma}

\theoremstyle{definition}

\newtheorem{thmwot}[thm]{}

\numberwithin{equation}{thm}

\newcommand{\Img}{\operatorname{Im}}

\newcommand{\End}{\operatorname{End}}
\newcommand{\Hom}{\operatorname{Hom}}
\newcommand{\Ind}{\operatorname{Ind}}
\newcommand{\Inf}{\operatorname{Inf}}

\newcommand{\Ker}{\operatorname{Ker}}

\newcommand{\Res}{\operatorname{Res}}

\newcommand{\set}[1]{\{\,#1\,\}}

\newcommand{\lsup}[1]{{^{#1}\hspace{-0.8pt}}}
\newcommand{\lsub}[1]{{_{#1}\hspace{-0.8pt}}}

\title{Non-injective inductions and restrictions of modules over finite groups
\footnote{supported by Sichuan Science and Technology Program (No. 2024NSFJQ0070) and by National Natural Science Foundation of China (No. 12271446).}}
\author{Conghui Li
\footnote{School of Mathematics, Southwest Jiaotong University, Chengdu 611756, China, Email: liconghui@swjtu.edu.cn},\quad
Yuting Tian
\footnote{School of Mathematics, Southwest Jiaotong University, Chengdu 611756, China, Email: 
tyt122814@my.swjtu.edu.cn}
\footnote{corresponding author}}
\date{}

\begin{document}

\setlength\abovedisplayskip{1ex plus 0.2ex minus 0.1ex}
\setlength\belowdisplayskip{1ex plus 0.2ex minus 0.1ex}

\raggedbottom
\maketitle

\begin{abstract}
In this note, we extend the inductions and restrictions of modules over finite groups to non-injective group homomorphisms, establishing transitivity, Frobenius reciprocity, Mackey's formula, etc.  


\end{abstract}

\section{Introduction}
Throughout this note, all groups are assumed to be finite and $R$ would always be a commutative ring with identity.

Induction and restriction are among the most basic and important constructions in group representation theory.
It is mainly used to study the relationship between the modules of groups and their subgroups.
For the definition and basic properties of induction and restrictions, see for example \cite[\S2.1,\S2.2,\S2.4]{Linck18a} or and \cite[\S4.3]{Web1}.

For any unitary $R$-algebra homomorphism $\alpha\colon A \to B$, we can consider the induction and restriction via $\alpha$.
Let $G$ be a group and $H \leq G$, then the induction from $H$ to $G$ and restriction from $G$ to $H$ can be viewed as special cases for the injection of $R$ algebras $RH \hookrightarrow RG$ induced by the injection of subgroup $H \hookrightarrow G$.

L. Puig introduced inductions of interior $G$-algebras in \cite{Puig81} and considered inductions of interior $G$-algebras under non-injective group homomorphisms in \cite[\S3]{Puig99}.
Puig established under some conditions for non-injective inductions of interior $G$-algebras the classical results, such as  transitivity, Frobenius reciprocity, Mackey's formula, etc.; see \cite{Linck1} for another method for the transitivity.

When $V$ is an $RG$-module, the endomorphism algebra $A=\End_R(V)$ is an interior $G$-algebra over $R$.
When $V$ is $R$-free, the (non-injective) inductions of modules over finite groups corresponding to (non-injective) inductions of endomorphism algebras; see \cite[Proposition 3.7]{Puig99}.
From this, some properties of (non-injective) inductions and restrictions of modules can be derived from those of (non-injective) inductions and restrictions of endomorphism algebras under the assumption that $V$ is $R$-free and some additional assumptions; see \cite[Remark 3.18]{Puig99}.

In this note, we give direct treatments of non-injective inductions and restrictions of modules over finite groups, such as transitivity, Frobenius reciprocity, Mackey formula, etc.
Our assumptions for those properties in \cite[Remark 3.18]{Puig99} are slightly weaker than those in \cite{Puig99}, and Puig did not indicate how the Mackey formula for non-injective inductions and restrictions of modules can be derived from the Mackey formula of non-injective inductions of interior $G$-algebras in \cite[Proposition 3.21]{Puig99}.
Considering possible applications, it may be not unnecessary to include detailed proof of these material.

After giving some notation, preliminaries and quoted results in \S\ref{not-lem}, we prove the generalization of the nature of non-injective induction and restriction in \S\ref{proof}.


\section{Definitions and some lemmas}\label{not-lem}

\begin{thmwot}\label{notation}
We first give basic definitions and notation.

Let $\varphi\colon G \to G_1$ be a group homomorphism, $V$ be an $RG$-module and $V_1$ be an $RG_1$-module.
Then $RG_1$ can be viewed as a (left or right) $RG$-module via $\varphi$.
Define the \textbf{non-injective induction} of $V$ via $\varphi$ by $\Ind_\varphi V = RG_1 \otimes_{RG}V$ as an $RG_1$-module and the \textbf{non-injective restriction} of $V_1$ via $\varphi$ is just $V_1$ viewed as an $RG$-module with $g\in G$ acting on it via $\varphi(g)$ and is denoted by $\Res_\varphi V_1$ or just $\lsub{\varphi} V_1$.
It is easy to see that $\Res_\varphi V_1 \cong RG_1\otimes_{RG_1}V_1 \cong \Hom_{RG_1}(RG_1,V_1)$ where $RG_1$ in the second term is viewed as an $(RG,RG_1)$-bimodule and $RG_1$ in the third term is viewed as an $(RG_1,RG)$-bimodule.

The above notion of non-injective induction and restriction is of course a special case of more general settings.
Let $A,B$ be $R$-algebras and $\alpha\colon A \to B$ is a unitary $R$-algebra homomorphism, $U$ be an $A$-module and $V$ be a $B$-module.
Then the induction of $U$ and restriction of $V$ via $\alpha$ are $\Ind_\alpha U \cong B \otimes_A U$ and $\Res_\alpha V$ with $a \in A$ acting on $V$ by $\alpha(a)$.

For any $R$-algebra $A$ and any two $A$-modules $U,V$, the notation $U \mid V$ means $U$ is isomorphic to a direct summand of $V$.
\end{thmwot}

The following lemma consider inductions and restrictions for surjections.
Recall that for any $RG$-module $V$, the module $R \otimes_{RG} V$ is the $R$-module of cofixed points of $V$, where $R$ is the right trivial $RG$-module.

\begin{lem}\label{lem-fixed-cofixed}
Assume $G$ is a finite group, $K\unlhd G$ and denote by $\pi\colon\, G \to G/K$ the natural surjection.
Let $V$ be an $RG$-module.
\begin{compactenum}[(1)]
\item
If $V$ is relatively $1$-projective, there is an isomorphism of $R$-modules:
\[ R\otimes_{RG}V \to V^G,\quad 1\otimes v \mapsto \sum_{g\in G}gv. \]
\item
There are isomorphism of $R(G/K)$-modules:
\[ \Ind_\pi V = R(G/K)\otimes_{RG}V\cong V/I(RK)V\cong R\otimes_{RK}V, \]
where $I(RK)$ is the augmentation idea of $RK$ and  $g\in G$ acts on $1\otimes v \in R \otimes_{RK} V$ as $g.(1\otimes v)=1\otimes gv$.
\item
If furthermore $\Res^G_KV$ is relatively $1$-projective, we also have
\[ \Ind_\pi V \cong R\otimes_{RK}V\cong V^K \]
as $R(G/K)$-modules.
\item
The functor $\Res_\pi-$ is just the functor of inflation $\Inf_{G/K}^G$, \emph{i.e.} identifying the $R(G/K)$-modules to those $RG$-modules with $K$ acting trivially.
\end{compactenum}
\end{lem}

\begin{proof}
(1) As V is relatively $1$-projective, there is an $R$-module $V_0$ such that $V \mid RG\otimes_R V_0$.
Set $U=RG\otimes_R V_0$, then $U \cong \Ind_1^GV_0$.
For $U$, we have isomorphisms of $R$-modules
\[ R\otimes_{RG}U = R\otimes_{RG}RG\otimes_RV_0 \cong V_0 \cong \left(\Ind_1^GV_0\right)^G \cong U^G, \]
where the isomorphism $V_0 \cong \left(\Ind_1^GV_0\right)^G$ is a special case of \cite[Proposition 2.5.7]{Linck18a}.
Tracking each step of the above, it can be seen that the isomorphism $R\otimes_{RG}U \cong U^G$ can be given by $1\otimes u \mapsto \sum_{g\in G}gu$.
Since $V$ is a direct summand of $U$ as $RG$-modules, the assertion in (1) follows.

(2) Follows from
\begin{align*}
& R(G/K)\otimes_{RG}V \cong RG/I(RK)RG \otimes_{RG}V \cong V/I(RK)RGV \\
=\ & V/I(RK)V \cong RK/I(RK) \cong R\otimes_{RK}V
\end{align*}
and the definition of the action of $G$ on $R\otimes_{RK}V$.

(3) It can be checked directly that the isomorphism of $R$-modules $R\otimes_{RK}V\cong V^K$ from (1) is compatible with the action of $G/K$.

(4) Follows immediately from the definition.
\end{proof}

We collect technical parts in the proof of Mackey formula in the following lemmas, which may be of independent interesting.

\begin{lem}\label{lem-Mackey-1}
Assume $\alpha\colon\, G \twoheadrightarrow H$ is a surjective group homomorphism, let $G_0$ be a subgroup of $G$ containing $K:=\Ker\alpha$ and denote by $H_0=\alpha(G_0)$ and by
\[ \alpha_0\colon\ G_0 \twoheadrightarrow H_0,\ g \mapsto \alpha(g) \]
the restriction of $\alpha$ to $G_0$ and $H_0$ as shown in the following diagram
\[ \begin{tikzcd}
G\arrow[r,"\alpha",two heads] & H \\
G_0\arrow[u,hook]\arrow[r,"\alpha_0",two heads] & H_0\arrow[u,hook] \\
K\arrow[u,hook]
\end{tikzcd} \]
\begin{compactenum}[(1)]
\item
For any $RG$-module $V$, we have $\Res^H_{H_0}\Ind_\alpha V \cong \Ind_{\alpha_0}\Res^G_{G_0} V$.
\item
For any $RH_0$-module $U$, we have $\Res_\alpha\Ind_{H_0}^H U \cong \Ind_{G_0}^G\Res_{\alpha_0} U$.
\end{compactenum}
\end{lem}

\begin{proof}
(1) It follows from part (2) of Lemma \ref{lem-fixed-cofixed} and the assumption $K=\Ker\alpha\leq G_0$ that there are isomorphisms of $RH_0$-modules
\[ \Res^H_{H_0}\Ind_\alpha V \cong R \otimes_{RK} V \cong \Ind_{\alpha_0}\Res^G_{G_0} V, \]
where only the action of $H_0$ is considered for the middle term $R\otimes_{RK}V$.

(2) By the definition of classical inductions and the assumption that $K=\Ker\alpha\leq G_0$, we have $R$-isomorphisms:
\[ \Ind_{H_0}^H U \cong \sum_{t\in[G/G_0]} \alpha(t) \otimes U \]
and
\[ \Ind_{G_0}^G \Res_{\alpha_0} U \cong \sum_{t\in[G/G_0]} t \otimes \lsub{\alpha_0}U. \]
Thus $\Res_\alpha\Ind_{H_0}^H U$ and $\Ind_{G_0}^G\Res_{\alpha_0} U$ are isomorphic as $R$-modules in the obvious way, and it can be checked directly that this $R$-isomorphism is compatible with the action of $G$.
\end{proof}

\begin{lem}\label{lem-Mackey-2}
Assume there are surjections of groups $\alpha,\beta,\gamma$ such that $\alpha=\beta\gamma$ as in the following commutative diagram
\[ \begin{tikzcd}
& A\arrow[rd,"\beta",two heads] &\\
B\arrow[rr,"\alpha",two heads]\arrow[ru,"\gamma",two heads] && C.
\end{tikzcd} \]
\begin{compactenum}[(1)]
\item
For any $RA$-module $V$, we have that $\Ind_\beta V \cong \Ind_\alpha\Res_\gamma V$.
\item
Assume $U$ is an $RB$-module such that $\Ker\alpha = \Ker\gamma \times K_0$ with $K_0$ acting trivially on $U$, then $\Ind_\gamma U \cong \Res_\beta\Ind_\alpha U$.
\end{compactenum}
\end{lem}

\begin{proof}
Denote by $K=\Ker\alpha$, $K_1=\Ker\gamma$ and $L=\Ker\beta$.

(1) It follows from part (2) of Lemma \ref{lem-fixed-cofixed} that there are isomorphisms of $RC$-modules
\[ \Ind_\beta V \cong R \otimes_{RL} V \]
and
\[ \Ind_\alpha\Res_\gamma V \cong R \otimes_{RK} \lsub{\gamma}V. \]
Note that $K/\Ker\gamma \cong L$ via $\gamma$ and $\Ker\gamma$ acts trivially on $\lsub{\gamma}V$, there is an isomorphism of $RC$-modules
\[ R \otimes_{RL} V \cong R \otimes_{RK} \lsub{\gamma}V. \]
Combining the above three isomorphisms gives the assertion.

(2) It follows from part (2) of Lemma \ref{lem-fixed-cofixed} that
\[ \Ind_\gamma U \cong R \otimes_{RK_1} U \]
and
\[ \Ind_\alpha U \cong R \otimes_{RK} U. \]
By the assumption on $U$ and $K$, we have an $R$-isomorphism
\[ R \otimes_{RK_1} U \cong R \otimes_{RK} U. \]
Since the action of $A$ on $R \otimes_{RK_1} U$ and the action of $C$ on $R \otimes_{RK} U$ are both induced by the action of $B$ on $U$, the above $R$-isomorphism is also compatible with the action of $A$, thus the assertion follows.
\end{proof}

\begin{lem}\label{lem-Mackey-3}
Assume there is a commutative diagram of group homomorphisms
\[ \begin{tikzcd}
H\arrow[r,"\alpha",two heads] & L\\
B\arrow[u,"\gamma"]\arrow[r,"\epsilon",two heads] & C\arrow[u,"i",hook]
\end{tikzcd} \]
with $\alpha$ and $\epsilon$ surjective and $i$ injective satisfying $\Ker\alpha \subseteq \Img\gamma$.
\begin{compactenum}[(1)]
\item
For any $RH$-module $V$, we have that $\Res_i\Ind_\alpha V \cong \Ind_\epsilon\Res_\gamma V$.
\item
Assume $U$ is an $RB$-module such that $\Ker\epsilon = \Ker\gamma \times K_0$ with $K_0$ acting trivially on $U$, then $\Ind_\gamma U \cong \Res_\alpha\Ind_i\Ind_\epsilon U$.
\end{compactenum}
\end{lem}

\begin{proof}
Decompose $\gamma$ as the canonical composition $\gamma=\iota\gamma_0$, where $\iota$ is the inclusion of $\Img\gamma$ into $H$ and $\gamma_0\colon\, B \twoheadrightarrow \Img\gamma$ the surjection induced by $\gamma$.
Since $i$ is injective, we have that $\Ker\gamma \subseteq \Ker\epsilon$.
Thus $\epsilon$ factors through a surjection $\bar{\epsilon}\colon\, \Img\gamma \twoheadrightarrow C$.
So we have that $\alpha\iota\gamma_0 = \alpha\gamma = i\epsilon = i\bar{\epsilon}\gamma_0$.
Since $\gamma_0$ is surjective, it follows that $\alpha\iota = i\bar{\epsilon}$.
Consequently, all the triangles and squares in the following diagram
\[ \begin{tikzcd}
H\arrow[rr,"\alpha",two heads] && L\\
&\Img\gamma\arrow[lu,"\iota",hook]\arrow[rd,"\bar{\epsilon}",two heads]& \\
B\arrow[uu,"\gamma"]\arrow[rr,"\epsilon",two heads]\arrow[ru,"\gamma_0",two heads] && C\arrow[uu,"i",hook]
\end{tikzcd} \]
commute.

(1) Viewing $C$ as a subgroup of $L$ via $i$, it follows from part (1) of Lemma \ref{lem-Mackey-1} that
\[ \Res_i\Ind_\alpha V \cong \Ind_{\bar{\epsilon}}\Res_\iota V. \]
Now it follows from part (1) of Lemma \ref{lem-Mackey-2} and transitivity of restrictions that
\[ \Ind_{\bar{\epsilon}}\Res_\iota V \cong \Ind_\epsilon\Res_{\gamma_0}\Res_\iota V \cong \Ind_\epsilon\Res_\gamma V. \]

(2) Viewing $C$ as a subgroup of $L$ via $i$, it follows from part (2) of Lemma \ref{lem-Mackey-1} that
\[ \Res_\alpha\Ind_i\Ind_\epsilon U \cong \Ind_\iota\Res_{\bar{\epsilon}}\Ind_\epsilon U. \]
Then it follows from part (2) of Lemma \ref{lem-Mackey-2} (note that $\Ker\gamma_0=\Ker\gamma$) and transitivity of inductions that
\[ \Ind_\iota\Res_{\bar{\epsilon}}\Ind_\epsilon U \cong \Ind_\iota\Ind_{\gamma_0} U \cong \Ind_\gamma U. \qedhere \]
\end{proof}

\section{Proof of the main result}\label{proof}

The transitivity holds even for the general induction and restriction for homomorphisms of $R$-algebras and this follows immediately from the definition.
We include this as below just for completeness.

\begin{thm}[tansitivity]\label{thm-transit}
Let $\alpha\colon A \to B$ and $\beta\colon B\to C$ be homomorphisms of $R$-algebras, $U$ be an $A$-module and $W$ be a $C$-module.
Then we have:
\begin{compactenum}[(1)]
  \item $\Res_\alpha(\Res_\beta W)\cong\Res_{\beta\alpha}W$;
  \item $\Ind_\beta(\Ind_\alpha U)\cong\Ind_{\beta\alpha}U$.
\end{compactenum}
\end{thm}

\begin{thm}\label{thm-ind-tensor}
Let $\varphi\colon G \to G_1$ be a group homomorphism, $V$ be an $RG$-module and $V_1$ be an $RG_1$-module.
Then there is an isomorphism of $RG_1$-modules
\[ \Ind_\varphi V \otimes_R V_1 \cong \Ind_\varphi (V\otimes_R\Res_\varphi V_1). \]
\end{thm}

\begin{proof}
This can be proved in the exactly the same way as for the classical induction and restriction by explicitly giving the mutually inverse isomorphisms
\[ (g_1\otimes v)\otimes v_1 \mapsto g_1\otimes(v\otimes g_1^{-1}v_1) \]
and
\[ g_1\otimes(v\otimes v_1) \mapsto (g_1\otimes v)\otimes g_1v_1. \]
It is readily to check that the above morphisms are well-defined homomorphisms of $RG_1$-modules and are mutually inverses.
\end{proof}

\begin{thm}
Let $\varphi\colon G \to G_1$ be a group homomorphism with $K=\Ker\varphi$, $V$ be an $RG$-module and $V_1$ be an $RG_1$-module.
\begin{compactenum}[(1)]
\item
$(\Res_\varphi V_1)^*\cong\Res_\varphi V_1^*$.
\item
If furthermore $\Res^G_KV$ is relatively $1$-projective, then
\[ (\Ind_\varphi V)^*\cong\Ind_\varphi V^*. \]
\end{compactenum}
\end{thm}

\begin{proof}
(1) follows from the definition of dual module and the restriction via $\varphi$.

(2) By the transitivity and the similar assertion for the classical induction, we may ssume that $\varphi$ is surjective.
In particular, $G_1 \cong G/K$.
Then by Lemma \ref{lem-fixed-cofixed}\,(2) and Hom-tensor adjoints, we have $RG_1$-isomorphisms:
\begin{align*}
& (\Ind_\varphi V)^* \cong \Hom_R(R\otimes_{RK}V,R) \\
\cong\ & \Hom_{RK}(V,\Hom_R(R,R)) \cong \Hom_{RK}(V,R) = (V^*)^K.
\end{align*}
Since $\Res^G_K V$ is relatively $1$-projective by assumption, it follows from the similar assertion for classical inductions that $V^*$ is also relatively $1$-projective.
Thus it follows from Lemma \ref{lem-fixed-cofixed}\,(3) for $V^*$ and the above isomorphisms that there are $RG_1$-isomorphisms:
\[\Ind_\varphi V^*\cong R\otimes_{RK}V^* \cong (V^*)^K \cong (\Ind_\varphi V)^*. \qedhere \]
\end{proof}

\begin{lem}\label{lem-hom-equal-tensor}
Let $\varphi\colon G \to G_1$ be a group homomorphism with $K=\Ker\varphi$ and $V$ be an $RG$-module such that $\Res^G_K V$ is relatively $1$-projective, then there is an isomorphism of $RG_1$-modules
\[ \Hom_{RG}(RG_1,V) \cong \Ind_\varphi V. \]
\end{lem}

\begin{proof}
Assume first that $\varphi$ is surjective, then $G_1 \cong G/K$.
Thus we have an isomorphism of $RG_1$-modules:
\[ \Hom_{RG}(RG_1,V) \to V^K,\quad \alpha \mapsto \alpha(1_{G_1}). \]
On the other hand, by Lemma \ref{lem-fixed-cofixed}, we have isomorphisms of $RG_1$-modules:
\[ \Ind_\varphi V \cong R(G/K) \otimes_{RG} V \cong R \otimes_{RK} V \cong V^K; \]
here the assumption that $\Res^G_K V$ is relatively $1$-projective is used.
Combining the above isomorphisms, the assertion when $\varphi$ is surjective follows.

Then we consider the general situation.
Denote by $\varphi_0\colon G \to \varphi(G)$ the homomorphism induced by $\varphi$.
Then using the transitivity, the similar assertion for the classical induction (see for example \cite[Corollary 4.3.8 (5)]{Web1}), the case in the previous paragraph and the Hom-Tensor adjoint, we have that
\begin{align*}
& \Ind_\varphi V \cong \Ind_{\varphi(G)}^{G_1}\Ind_{\varphi_0} V \cong \Hom_{R\varphi(G)}(RG_1,\Ind_{\varphi_0}V) \\
\cong\ & \Hom_{R\varphi(G)}\left( RG_1, \Hom_{RG}(R\varphi(G),V) \right) \cong \Hom_{RG}( R\varphi(G)\otimes_{R\varphi(G)}RG_1, V ) \\
\cong\ & \Hom_{RG}(RG_1,V).
\end{align*}
\end{proof}

\begin{thm}[Frobenius reciprocity]
Let $\varphi\colon G \to G_1$ be a group homomorphism with $K=\Ker\varphi$, $V$ be an $RG$-module and $V_1$ be an $RG_1$-module.
\begin{compactenum}[(1)]
\item
$ \Hom_{RG_1}(\Ind_\varphi V,V_1) \cong \Hom_{RG}(V,\Res_\varphi V_1) $;
\item
If furthermore $\Res^G_KV$ is relatively $1$-projective, then\[\Hom_{RG_1}(V_1,\Ind_\varphi V)\cong\Hom_{RG}(\Res_\varphi V_1,V).\]
\end{compactenum}
\end{thm}

\begin{proof}
(1) Follows from the classical Hom-Tensor adjoint:
\begin{align*}
& \Hom_{RG_1}(\Ind_\varphi V,V_1) = \Hom_{RG_1}(RG_1\otimes_{RG}V,V_1) \\
\cong\ & \Hom_{RG}(V,\Hom_{RG_1}(RG_1,V_1)) \cong \Hom_{RG}(V,\Res_\varphi V_1).
\end{align*}
(2) Follows from the classical Hom-Tensor adjoint and Lemma \ref{lem-hom-equal-tensor}:
\begin{align*}
& \Hom_{RG_1}(V_1,\Ind_\varphi V) \cong \Hom_{RG_1}(V_1,\Hom_{RG_1}(RG_1,V) \\
\cong\ & \Hom_{RG}(RG_1\otimes_{RG_1}V_1,V) \cong \Hom_{RG}(\Res_\varphi V_1,V).
\end{align*}
\end{proof}

\begin{thm}[Mackey's formula]\label{mackey}
Let $\alpha\colon\, K \to G$, $\beta\colon\, H \to G$ be two group homomorphism.
For any $x\in G$, set
\[ \lsup{x}\alpha\colon\ K \to G,\ k \mapsto x\alpha(k)x^{-1}, \]
and consider the pull back of $\lsup{x}\alpha$ and $\beta$ (see for example \cite[Example 2.1.16(2)]{Agore}):
\[\begin{tikzcd}
K\arrow[r,"\lsup{x}\alpha"]&G\\
B_x\arrow[u,"\gamma_x"]\arrow[r,"\delta_x"]&H\arrow[u,"\beta"]
\end{tikzcd}\]
Then we have for any $K$-module $V$ that
\[ \Res_\beta\Ind_\alpha V\cong\bigoplus_{x\in[\beta(H)\backslash G/\alpha(K)]}\Ind_{\delta_x}\Res_{\gamma_x}V. \]
\end{thm}

\begin{proof}

Denote $\alpha_0\colon K \to \alpha(K)$ and $\beta_0\colon H \to \beta(H)$ the induced map by $\alpha$ and $\beta$ respectively.
Consider the following diagram:
\[\begin{tikzcd}
K\arrow[rrr,"\lsup{x}\alpha"]\arrow[rd,"\lsup{x}\alpha_0",two heads]& & &G\\
&\lsup{x}\alpha(K)\arrow[rru,"a",hook]\\
&\beta(H)\bigcap \lsup{x}\alpha(K)\arrow[u,"i_x",hook]\arrow[r,"j_x",hook]&\beta(H)\arrow[ruu,"b",hook]\\
B_x\arrow[uuu,"\gamma_x"]\arrow[ru,"\epsilon_x",two heads]\arrow[rrr,"\delta_x"]&&&H\arrow[lu,"\beta_0",two heads]\arrow[uuu,"\beta"]
\end{tikzcd}
\]
where $\epsilon_x$ exists because $(i_x,j_x)$ is the pull back of $(a,b)$.
Precisely,
\[ B_x = \set{(k,h)\in K\times H \mid \lsup{x}\alpha(k)=\beta(h)}, \]
\[ \epsilon_x\colon\ B_x \to \beta(H)\cap\lsup{x}\alpha(K),\ (k,h) \mapsto \beta(h)=\lsup{x}\alpha(k). \]
Thus all the triangles and squares in the above diagram commute.

Using the transitivity and the classical Mackey formula, we have
\begin{equation}\label{equ-1}
\begin{aligned}
& \Res_\beta\Ind_\alpha V \cong \Res_{\beta_0}\Res^G_{\beta(H)}\Ind^G_{\alpha(K)}\Ind_{\alpha_0} V\\
\cong\ & \bigoplus_{x\in[\beta(H)\backslash G/\alpha(K)]}\Res_{\beta_0}\Ind^{\beta(H)}_{\beta(H)\bigcap \lsup{x}\alpha(K)}\Res^{\lsup{x}\alpha(K)}_{\beta(H)\bigcap \lsup{x}\alpha(K)} {\lsup{x}\Ind_{\alpha_0}} V \\
\cong\ & \bigoplus_{x\in[\beta(H)\backslash G/\alpha(K)]}\Res_{\beta_0}\Ind_{j_x}\Res_{i_x}\Ind_{\lsup{x}\alpha_0}V 
\end{aligned}
\end{equation}

It can be checked easily that $\Ker\lsup{x}\alpha_0 \leq \Img\gamma_x$, then by applying part (1) of Lemma \ref{lem-Mackey-3} to the following diagram
\[ \begin{tikzcd}
K\arrow[r,"\lsup{x}\alpha_0",two heads] & \lsup{x}\alpha(K)\\
B_x\arrow[u,"\gamma_x"]\arrow[r,"\epsilon_x",two heads] & \beta(H) \cap \lsup{x}\alpha(K)\arrow[u,"i_x",hook]
\end{tikzcd} \]
we have that
\begin{equation}\label{equ-2}
\Res_{i_x}\Ind_{\lsup{x}\alpha_0} V \cong \Ind_{\epsilon_x}\Res_{\gamma_x} V.
\end{equation}

It can be checked similarly that $\Ker\beta_0 \leq \Img\delta_x$.
It can be calculated directly that
\begin{align*}
\Ker\gamma_x &= \set{ (1,h) \mid \beta(h)=1 }, \\
\Ker\delta_x &= \set{ (k,1)\mid \lsup{x}\alpha(k)=1 }, \\
\Ker\epsilon_x &= \set{ (k,h) \mid \beta(h) = \lsup{x}\alpha(k) = 1 }.
\end{align*}
Thus, $\Ker\epsilon_x=\Ker\delta_x\times\Ker\gamma_x$.
Also, by definition, $\Ker\gamma_x$ acts trivially on $\Res_{\gamma_x} V$.
So by applying part (2) of Lemma \ref{lem-Mackey-3} to the following diagram
\[ \begin{tikzcd}
H\arrow[r,"\beta_0",two heads] & \beta(H)\\
B_x\arrow[u,"\delta_x"]\arrow[r,"\epsilon_x",two heads] & \beta(H) \cap \lsup{x}\alpha(K)\arrow[u,"j_x",hook]
\end{tikzcd} \]
we have that
\begin{equation}\label{equ-3}
\Res_{\beta_0}\Ind_{j_x}\Ind_{\epsilon_x}\Res_{\gamma_x} V \cong \Ind_{\delta_x}\Res_{\gamma_x} V.
\end{equation}

Finally, the assertion follows by combining (\ref{equ-1}), (\ref{equ-2}) and (\ref{equ-3}).
\end{proof}


\end{document}